\documentclass[12pt]{amsart}

\usepackage{amssymb}
\usepackage{amsfonts}
\usepackage{amsmath}
\usepackage[colorlinks=true,urlcolor=blue,
citecolor=red,linkcolor=blue,linktocpage,pdfpagelabels,
bookmarksnumbered,bookmarksopen]{hyperref}

\DeclareMathOperator*{\essinf}{ess ~inf}         

\newtheorem{theorem}{Theorem}[section]
\newtheorem{proposition}{Proposition}[section]
\newtheorem{lemma}[theorem]{Lemma}
\theoremstyle{definition}

\theoremstyle{remark}
\newtheorem{remark}[theorem]{Remark}
\numberwithin{equation}{section}

\newcommand{\R}{{\mathbb R}}
\newcommand{\N}{{\mathbb N}}

\newcommand{\eps}{\varepsilon}
\newcommand{\cal}{\mathcal}

\begin{document}
\title[Singular quasilinear systems ]{Singular quasilinear elliptic systems in $\R^{N}$}
\subjclass[2010]{35J75; 35J48; 35J92}
\keywords{Singular elliptic system; $p$-Laplacian; Schauder's fixed point theorem; a priori estimate.}
\begin{abstract}
The existence of positive weak solutions to a singular quasilinear elliptic system in the whole space is established via suitable a priori estimates and Schauder's fixed point theorem.
\end{abstract}

\author{S.A. Marano}
\address{Salvatore A. Marano\\
Dipartimento di Matematica e Informatica\\
Universit\`a degli Studi di Catania \
Viale A. Doria 6, 95125 Catania, Italy}
\email{marano@dmi.unict.it}
\author{G. Marino}
\address{Greta Marino\\
Dipartimento di Matematica e Informatica \\
Universit\`a degli Studi di Catania \\
Viale A. Doria 6, 95125 Catania, Italy}
\email{greta.marino@dmi.unict.it}
\author{A. Moussaoui}
\address{Abdelkrim Moussaoui\\
Biology Department\\
A. Mira Bejaia University\\
Targa Ouzemour, 06000 Bejaia, Algeria}
\email{abdelkrim.moussaoui@univ-bejaia.dz}
\maketitle
\section{Introduction}
In this paper, we consider the following system of quasilinear elliptic equations:
\begin{equation}\label{p}
\left\{ 
\begin{array}{ll}
-\Delta_{p_1} u=a_1(x)f(u,v) & \text{in } \R^N, \\
-\Delta_{p_2} v=a_2(x)g(u,v) & \text{in } \R^N, \\ 
u,v>0 & \text{in } \R^N,
\end{array}
\right. \tag{{\rm P}} 
\end{equation}
where $N\geq 3$,  $1<p_i<N$, while $\Delta_{p_i}$ denotes the $p_i$-Laplace differential operator. Nonlinearities $f,g:\R^+\times\R^+\to\R^+$ are continuous and fulfill the condition
\begin{itemize}
\item[$({\rm H}_{f,g})$] There exist $m_i,M_i>0$, $i=1,2$, such that
\begin{equation*}
m_1s^{\alpha_1}\leq f(s,t)\leq M_1s^{\alpha_1}(1+t^{\beta_1}),
\end{equation*}
\begin{equation*}
m_2t^{\beta _2}\leq g(s,t)\leq M_2(1+s^{\alpha _2})t^{\beta_2}
\end{equation*}
for all $s,t\in\R^+$, with $-1<\alpha_1,\beta_2<0<\alpha_2,\beta_1$,
\begin{equation}\label{c2}
\alpha_1+\alpha_2<p_1-1,\;\;\beta_1+\beta_2<p_2-1,
\end{equation}
as well as
\begin{equation*}
\beta_1<\frac{p_2^*}{p_1^*}\min\{p_1-1, p_1^*-p_1\},\;\;
\alpha_2<\frac{p_1^*}{p_2^*}\min\{p_2-1, p_2^*-p_2\}.
\end{equation*}
\end{itemize}
Here, $p^*_i$ denotes the critical Sobolev exponent corresponding to $p_i$, namely $p^*_i:=\frac{Np_i}{N-p_i}$. Coefficients $a_i:\R^N\to\R$ satisfy the assumption
\begin{itemize}
\item[$({\rm H}_a)$] $a_i(x)>0$ a.e. in $\R^N$ and $a_i\in L^1(\R^N)\cap L^{\zeta_i}(\R^N)$, where
\begin{equation*}
\frac{1}{\zeta_1}\leq 1-\frac{p_1}{p_1^*}-\frac{\beta_1}{p_2^*}\, ,\;\;\frac{1}{\zeta_2}\leq 1- \frac{p_2}{p_2^*}- 
\frac{\alpha_2}{p_1^*}\, .
\end{equation*}
\end{itemize}
Let $\cal{D}^{1,p_i}(\R^N)$ be the closure of $C_0^\infty(\R^N)$ with respect to the norm
$$\Vert w\Vert_{\cal{D}^{1,p_i}(\R^N)}:=\Vert\nabla w\Vert_{L^{p_i}(\R^N)}.$$
Recall \cite[Theorem 8.3]{LL} that
\begin{equation*}
\cal{D}^{1,p_i}(\R^N)=\{w\in L^{p_i^*}(\R^N):|\nabla w|\in L^{p_i}(\R^N)\}.
\end{equation*}
Moreover, if $w\in\cal{D}^{1,p_i}(\R^N)$ then $w$ vanishes at infinity, i.e., the set $\{x\in\R^N: w(x)>k\}$ has finite measure for all $k>0$; see \cite[p. 201]{LL}.

A pair $(u,v)\in \cal{D}^{1,p_1}(\R^N)\times \cal{D}^{1,p_2}(\R^{N})$ is called a (weak) solution to \eqref{p} provided $u,v>0$ a.e. in $\R^N$ and
\begin{equation*} 
\left\{ 
\begin{array}{ll}
\int_{\R^N}|\nabla u|^{p_{1}-2}\nabla u\nabla \varphi\, dx & =\int_{\R^N}a_{1}f(u,v)\varphi\, dx,\\
\phantom{}\\
\int_{\R^N}|\nabla v|^{p_{2}-2}\nabla v\nabla \psi\, dx & =\int_{\R^N}a_{2}g(u,v)\psi\, dx
\end{array}
\right.
\end{equation*}
for every $(\varphi ,\psi )\in \cal{D}^{1,p_{1}}(\R^N)\times \cal{D}^{1,p_{2}}(\R^N)$. 

The most interesting aspect of the work probably lies in the fact that both $f$ and $g$ can exhibit singularities through $\R^{N}$, which, without loss of generality, are located at zero. Indeed, $-1<\alpha_1,\beta_2<0$ by $({\rm H}_{f,g})$. It represents a serious difficulty to overcome, and is rarely handled in the literature.

As far as we know, singular systems in the whole space have been investigated only for $p:=q:=2$, essentially exploiting the linearity of involved differential operators. In such a context, \cite{DKC, DKW, MKT} treat the so called Gierer-Meinhardt system, that arises from the mathematical modeling of important biochemical processes. Nevertheless, even in the semilinear case, \eqref{p} cannot be reduced to Gierer-Meinhardt's case once $({\rm H}_{f,g})$ is assumed. The situation looks quite different when a bounded domain takes the place of $\R^N$: many singular systems fitting the framework of \eqref{p} have been studied, and meaningful contributions are already available \cite{AC,GHS,GST,EPS,G1,G2,HMV,MS,MM1,MM2,MM3}.

Here, variational methods do not work, at least in a direct way, because the Euler functional associated with problem \eqref{p} is not well defined. A similar comment holds for sub-super-solution techniques, that are usually employed in the case of bounded domains. Hence, we were naturally led to apply fixed point results. An a priori estimate in $L^\infty(\R^N)\times L^\infty(\R^N)$ for solutions of \eqref{p} is first established (cf. Theorem \ref{T3}) by a Moser's type iteration procedure and an adequate truncation, which, due to singular terms, require a specific treatment. We next perturb \eqref{p} by introducing a parameter $\eps>0$. This produces the family of regularized systems
\begin{equation} \label{pr}
\left\{ 
\begin{array}{ll}
-\Delta_{p_1} u=a_1(x)f(u+\eps,v) & \text{in }\R^{N}, \\ 
-\Delta_{p_2} v=a_2(x)g(u,v+\eps) & \text{in }\R^{N}, \\ 
u,v>0 & \text{in }\R^{N},
\end{array}
\right.  \tag{${\rm P}_\eps$} 
\end{equation}
whose study yields useful information on the  original problem. In fact, the previous $L^\infty$- boundedness still holds for solutions to \eqref{pr}, regardless of $\eps$. Thus, via Schauder's fixed point theorem, we get a solution $(u_\eps,v_\eps)$ lying inside a rectangle given by positive lower bounds, where $\eps$ does not appear, and positive upper bounds, that may instead depend on $\eps$. Finally, letting $\eps\to 0^+$ and using the $({\rm S})_+$-property of the negative $p$-Laplacian in $\cal{D}^{1,p}(\R^N)$ (see Lemma \ref{L3}) yields a weak solution to \eqref{p}; cf. Theorem \ref{mainthm}

The rest of this paper is organized as follows. Section \ref{S2} deals with preliminary results. The a priori estimate of solutions to \eqref{p} is proven in Section \ref{S6}, while the next one treats system \eqref{pr}. Section \ref{S4} contains our existence result for problem \eqref{p}.
\section{Preliminaries}\label{S2}

Let $\Omega\subseteq\R^N$ be a measurable set, let $t\in\R$, and let $w,z\in L^p(\R^N)$. We write $m(\Omega)$ for the Lebesgue measure of $\Omega$, while $t^\pm:=\max\{\pm t,0\}$, $\Omega(w\leq t):=\{x\in\Omega: w(x)\leq t\}$, $\Vert w\Vert_p:=\Vert w\Vert_{L^p(\R^N)}$. The meaning of $\Omega(w>t)$, etc. is analogous. By definition, $w\leq z$ iff $w(x)\leq z(x)$ a.e. in $\R^N$.

Given $1\leq q<p$, neither $L^p(\R^N)\hookrightarrow L^q(\R^N)$ nor the reverse embedding hold true. However, the situation looks better for functions belonging to $L^1(\R^N)$. Indeed,
\begin{proposition}\label{reg}
Suppose $p>1$ and $w\in L^1(\R^N)\cap L^p(\R^N)$. Then $w\in L^q(\R^N)$ whatever $q\in\ ]1,p[$.
\end{proposition}
\begin{proof}
Thanks to H\"{o}lder's inequality, with exponents $p/q$ and $p/(p-q)$, and Chebyshev's inequality one has 
\begin{equation*}
\begin{split}
\Vert w & \Vert_q^q=\int_{\R^N(|w|\leq 1)}|w|^qdx+\int_{\R^N(|w|>1)}|w|^qdx\\
&\leq \int_{\R^N(|w|\leq 1)}|w|\,dx+\left(\int_{\R^N(|w|>1)}|w|^pdx\right)^{q/p} [m(\R^N(|w|>1))]^{1-q/p}\\
& \leq \int_{\R^N}|w|\,dx+\left(\int_{\R^N}|w|^pdx\right)^{q/p}\left(\int_{\R^N}|w|^pdx\right)^{1-q/p}\\
& =\Vert w\Vert_1+\Vert w\Vert _p^p.
\end{split}
\end{equation*}
This completes the proof.
\end{proof}
The summability properties of $a_i$ collected below will be exploited throughout the paper.
\begin{remark}\label{R1}
Let assumption $({\rm H}_a)$ be fulfilled. Then, for any $i=1,2$,
\begin{itemize}
\item[$({\rm j}_1)$] $a_i\in L^{(p_i^*)'}(\R^N)$.
\item[$({\rm j}_2)$] $a_i\in L^{\gamma_i}(\R^N)$, where $\gamma_i:=1/(1-t_i)$, with
$$t_1:=\frac{\alpha_1+1}{p_1^*}+\frac{\beta_1}{p_2^*},\quad t_2:=\frac{\alpha_2}{p_1^*}+\frac{\beta_2+1}{p_2^*}.$$
\item[$({\rm j}_3)$]$a_i\in L^{\delta_i}(\R^N)$, for $\delta_i:=1/(1-s_i)$ and
$$s_1:=\frac{\alpha_1+1}{p_1^*},\quad s_2:=\frac{\beta_2+1}{p_2^*}.$$
\item[$({\rm j}_4)$] $a_i\in L^{\xi_i}(\R^N)$, where  $\xi_i\in\ ]p_i^*/(p_i^*-p_i),\zeta_i[$.
\end{itemize}
To verify $({\rm j}_1)$--$({\rm j}_4)$ we simply note that $\zeta_i >\max\{(p_i^*)',\gamma_i,\delta_i,\xi_i\}$ and apply Proposition \ref{reg}.
\end{remark}
Let us next show that the operator $-\Delta_p$ is of type $({\rm S})_+$ in $\cal{D}^{1,p}(\R^N)$.

\begin{proposition}\label{Sprop}
If $1<p<N$ and $\{u_n\}\subseteq\cal{D}^{1,p}(\R^N)$ satisfies
\begin{equation}\label{t1}
u_n\rightharpoonup u\;\;\text{in}\;\;\cal{D}^{1,p}(\R^N),  
\end{equation}
\begin{equation}\label{t2}
\limsup_{n\to\infty}\left\langle -\Delta_p u_n,u_n-u\right\rangle \leq 0,  
\end{equation}
then $u_n\to u$ in $\cal{D}^{1,p}(\R^N)$.
\end{proposition}
\begin{proof}
By monotonicity one has
\begin{equation*}
\left\langle-\Delta_p u_n-(-\Delta_p u),u_n-u\right\rangle \geq 0\quad\forall\, n\in\mathbb{N},
\end{equation*}
which evidently entails
\begin{equation*}
\liminf_{n\to\infty}\left\langle -\Delta_p u_n-(-\Delta_p u),u_n-u\right\rangle \geq 0.  
\end{equation*}
Via \eqref{t1}--\eqref{t2} we then get
\begin{equation*}
\limsup_{n\to\infty}\left\langle-\Delta_p u_n-(-\Delta_p u),u_n-u\right\rangle \leq 0.
\end{equation*}
Therefore,
\begin{equation}\label{00}
\lim_{n\to\infty }\int_{\R^N}\left(|\nabla u_n|^{p-2}\nabla u_n-|\nabla u|^{p-2}\nabla u\right)(\nabla u_n-\nabla u)\, dx=0.  
\end{equation}
Since \cite[Lemma A.0.5]{P} yields
\begin{equation*}
\begin{split}
& \int_{\R^N}\left(|\nabla u_n|^{p-2}\nabla u_n-|\nabla u|^{p-2}\nabla u\right)(\nabla u_n-\nabla u)\, dx \\ 
& \geq \left\{ 
\begin{array}{ll}
C_p\int_{\R^N}\frac{\vert \nabla (u_n-u)\vert^2}{(|\nabla u_n|+|\nabla u|)^{2-p}}\, dx  & \text{  if }1<p<2,\\
\phantom{}\\
C_{p}\int_{\R^N}\vert\nabla (u_n-u)\vert^p\, dx & \text{otherwise}
\end{array}%
\right.\quad\forall\, n\in\mathbb{N},
\end{split}
\end{equation*}
the desired conclusion, namely
\begin{equation*}
\lim_{n\to\infty}\int_{\R^N}|\nabla (u_n-u)|^p\, dx=0,
\end{equation*}
directly follows from \eqref{00} once $p\geq 2$. If $1<p<2$ then H\"{o}lder's inequality and \eqref{t1} lead to
\begin{equation*}
\begin{split}
& \int_{\R^N}|\nabla (u_n-u)|^p\, dx\\
& =\int_{\R^N}\frac{|\nabla (u_n-u)|^p}{(|\nabla u_n|+|\nabla u|)^{\frac{p(2-p)}{2}}}\,
(|\nabla u_n|+|\nabla u|)^{\frac{p(2-p)}{2}} dx\\
& \leq\left(\int_{\R^N}\frac{|\nabla (u_n-u)|^2}{(|\nabla u_n|+|\nabla u|)^{2-p}}\, dx\right)^{\frac{p}{2}}
\left(\int_{\R^N}(|\nabla u_n|+|\nabla u|)^p dx\right)^{\frac{2-p}{2}}\\
& \leq C\left(\int_{\R^N}\frac{|\nabla (u_n-u)|^2}{(|\nabla u_n|+|\nabla u|)^{2-p}}\, dx\right)^{\frac{p}{2}},\quad n\in\mathbb{N},
\end{split}
\end{equation*}
with appropriate $C>0$. Now, the argument goes on as before.
\end{proof}
\section{Boundedness of solutions}\label{S6}
The main result of this section, Theorem \ref{T3} below, provides an $L^\infty(\R^N)$ - a priori estimate for weak solutions to  \eqref{p}. Its proof will be performed into three steps.
\begin{lemma}[$L^{p_i^*}(\R^N)$ - uniform boundedness] 
There exists $\rho>0$ such that
\begin{equation}\label{apin1}
\max\left\{\Vert u\Vert_{p_1^*},\,\Vert v\Vert_{p_2^*}\right\}\leq\rho
\end{equation}
for every $(u,v)\in \cal{D}^{1,p_1}(\R^N)\times\cal{D}^{1,p_2}(\R^N)$ solving problem \eqref{p}.
\end{lemma}
\begin{proof}
Multiply both equations in \eqref{p} by $u$ and $v$, respectively, integrate over $\R^N$, and use $({\rm H}_{f,g})$ to arrive at
\begin{equation*}
\begin{split}
\Vert\nabla u\Vert_{p_1}^{p_1}=\int_{\R^N} a_1 f(u,v)u\, dx\leq M_1\int_{\R^N} a_1 u^{\alpha_1+1}(1+v^{\beta_1})\, dx,\\
\Vert\nabla v\Vert_{p_2}^{p_2}=\int_{\R^N} a_2 g(u,v)v\, dx\leq M_2\int_{\R^N} a_2 (1+u^{\alpha_2})v^{\beta_2+1}\, dx.
\end{split}
\end{equation*}
Through the embedding $\cal{D}^{1,p_i}(\R^N)\hookrightarrow L^{p_i^*}(\R^N)$, besides H\"{o}lder's inequality, we obtain
\begin{equation*}
\begin{split}
\Vert\nabla u\Vert_{p_1}^{p_1} & \leq M_1\left(\Vert a_1\Vert _{\delta_1}\Vert u\Vert _{p_1^*}^{\alpha _1+1}
+\Vert a_1\Vert_{\gamma_1}\Vert u\Vert_{p_1^*}^{\alpha_1+1}\Vert v\Vert_{p_2^*}^{\beta_1}\right) \\
& \leq C_1\Vert \nabla u\Vert_{p_1}^{\alpha_1+1}\left(\Vert a_1\Vert_{\delta_1}+\Vert a_1\Vert_{\gamma_1}
\Vert\nabla v\Vert_{p_2}^{\beta_1}\right);
\end{split}
\end{equation*}
cf. also Remark \ref{R1}. Likewise,
\begin{equation*}
\Vert\nabla v\Vert_{p_2}^{p_2}\leq C_2\Vert\nabla v\Vert_{p_2}^{\beta_2+1}\left(\Vert a_2\Vert_{\delta_2}
+\Vert a_2\Vert_{\gamma_2}\Vert\nabla u\Vert_{p_1}^{\alpha_2}\right).
\end{equation*}
Thus, a fortiori,
\begin{equation}\label{grad}
\begin{split}
\Vert\nabla u\Vert_{p_1}^{p_1-1-\alpha_1} & \leq C_1\left(\Vert a_1\Vert_{\delta _1}+\Vert a_1\Vert_{\gamma_1}
\Vert\nabla v\Vert_{p_2}^{\beta_1}\right),\\
\Vert \nabla v\Vert_{p_2}^{p_2-1-\beta_2} & \leq C_2\left(\Vert a_2\Vert _{\delta_2}+\Vert a_2\Vert_{\gamma_2}
\Vert \nabla u\Vert_{p_1}^{\alpha_2}\right),
\end{split}
\end{equation}
which imply
\begin{equation*}
\begin{split}
& \Vert\nabla u\Vert_{p_1}^{p_1-1-\alpha_1}+\Vert \nabla v\Vert_{p_2}^{p_2-1-\beta_2}\\
& \leq C_1\left(\Vert a_1\Vert_{\delta _1}+\Vert a_1\Vert_{\gamma_1}\Vert\nabla v\Vert_{p_2}^{\beta_1}\right)
+C_2\left(\Vert a_2\Vert_{\delta_2}+\Vert a_2\Vert_{\gamma_2}\Vert\nabla u\Vert_{p_1}^{\alpha_2}\right).
\end{split}
\end{equation*}
Rewriting this inequality as
\begin{equation}\label{use1}
\begin{split}
\Vert\nabla u\Vert_{p_1}^{\alpha_2} & \left(\Vert\nabla u\Vert_{p_1}^{p_1-1-\alpha_1-\alpha_2}-C_2\Vert a_2\Vert_{\gamma_2}\right)\\
& +\Vert\nabla v\Vert_{p_2}^{\beta_1}\left(\Vert\nabla v\Vert_{p_2}^{p_2-1-\beta_1-\beta_2}-C_1\Vert a_1\Vert_{\gamma_1}\right) \\
& \leq C_1\Vert a_1\Vert_{\delta_1}+C_2\Vert a_2\Vert_{\delta_2},
\end{split}
\end{equation}
four situations may occur. If
$$\Vert\nabla u \Vert_{p_1}^{p_1-1-\alpha_1-\alpha_2}\leq C_2\Vert a_2 \Vert_{\gamma_2}\, ,\quad
\Vert \nabla v\Vert_{p_2}^{p_2-1-\beta_1-\beta_2}\leq C_1\Vert a_1\Vert_{\gamma_1}$$
then \eqref{apin1} follows from $({\rm j}_2)$ of Remark \ref{R1}, conditions \eqref{c2}, and the embedding $\cal{D}^{1,p_i}(\R^N)\hookrightarrow L^{p_i^*}(\R^N)$. Assume next that 
\begin{equation}\label{use2}
\Vert \nabla u\Vert_{p_1}^{p_1-1-\alpha_1-\alpha_2}>C_2\Vert a_2\Vert_{\gamma_2}\, ,\quad
\Vert \nabla v\Vert_{p_2}^{p_2-1-\beta_1-\beta_2}>C_1\Vert a_1\Vert_{\gamma_1}\, .
\end{equation}
Thanks to \eqref{use1} one has
\begin{equation*}
\Vert\nabla u\Vert_{p_1}^{\alpha_2}(\Vert\nabla u\Vert_{p_1}^{p_1-1-\alpha_1-\alpha_2}-C_2\Vert a_2\Vert_{\gamma_2})\leq C_1\Vert a_1\Vert _{\delta_1}+C_2\Vert a_2\Vert_{\delta_2},
\end{equation*}
whence, on account of \eqref{use2},
\begin{equation*}
\begin{split}
\Vert\nabla u\Vert_{p_1}^{p_1-1-\alpha_1-\alpha_2} & 
\leq \frac{C_1\Vert a_1\Vert_{\delta_1}+C_2\Vert a_2\Vert_{\delta_2}}{\Vert\nabla u\Vert_{p_1}^{\alpha_2}}
+C_2\Vert a_2\Vert_{\gamma_2} \\
& \leq\frac{C_1\Vert a_1\Vert_{\delta_1}+C_2\Vert a_2\Vert_{\delta _2}}{\Vert a_2\Vert_{\gamma_2}^{\frac{\alpha_2}{p_1-1-\alpha_1-\alpha_2}}}+C_2\Vert a_2\Vert_{\gamma_2}.
\end{split}
\end{equation*}
A similar inequality holds true for $v$. So, \eqref{apin1} is achieved reasoning as before. Finally, if
\begin{equation}\label{use3}
\Vert\nabla u\Vert_{p_1}^{p_1-1-\alpha_1-\alpha_2}\leq C_2\Vert a_2 \Vert_{\gamma_2}\, ,\quad
\Vert\nabla v\Vert_{p_2}^{p_2-1-\beta_1-\beta_2}> C_1 \Vert a_1\Vert_{\gamma_1}
\end{equation}
then \eqref{grad} and \eqref{use3} entail
$$\Vert\nabla v\Vert_{p_2}^{p_2-1-\beta_2}\leq C_2\left[ \Vert a_2 \Vert_{\delta_2}+ \Vert a_2 \Vert_{\gamma_2}
\left( C_2\Vert a_2\Vert_{\gamma_2}\right)^{\frac{\alpha_2}{p_1-1-\alpha_1-\alpha_2}}\right].$$
By \eqref{c2} again we thus get
$$\max\{\Vert\nabla u\Vert_{p_1}, \Vert\nabla v\Vert_{p_2}\}\leq C_3\, ,$$
where $C_3>0$. This yields \eqref{apin1}, because $\cal{D}^{1,p_i}(\R^N)\hookrightarrow L^{p_i^*}(\R^N)$. The last case, i.e., 
\begin{equation*}
\Vert\nabla u\Vert_{p_1}^{p_1-1-\alpha_1-\alpha_2}>C_2\Vert a_2 \Vert_{\gamma_2}\, ,\quad
\Vert\nabla v\Vert_{p_2}^{p_2-1-\beta_1-\beta_2}\leq C_1 \Vert a_1\Vert_{\gamma_1}
\end{equation*}
is analogous.
\end{proof}
To shorten notation, write
$${\cal D}^{1,p_i}(\R^N)_+:=\{ w\in{\cal D}^{1,p_i}(\R^N): w\geq 0\;\text{ a.e. in }\;\R^N\}.$$
\begin{lemma}[Truncation] Let $(u,v)\in \cal{D}^{1,p_1}(\R^N)\times\cal{D}^{1,p_2}(\R^N)$ be a weak solution of \eqref{p}. Then
\begin{equation}\label{trunc1}
\int_{\R^N(u>1)}\vert\nabla u\vert ^{p_1-2}\nabla u\nabla\varphi\, dx\leq M_1\int_{\R^N(u>1)} a_1(1+v^{\beta_1})\varphi\, dx,
\end{equation}
\begin{equation}\label{trunc2}
\int_{\R^N(v>1)}\vert\nabla v\vert^{p_2-2}\nabla v\nabla\psi\, dx\leq M_2\int_{\R^N(v>1)} a_2(1+u^{\alpha_2})\psi\, dx
\end{equation}
for all $(\varphi,\psi)\in\cal{D}^{1,p_1}(\R^N)_+\times\cal{D}^{1,p_2}(\R^N)_+$.
\end{lemma}
\begin{proof}
Pick a $C^1$ cut-off function $\eta:\R\to [0,1]$ such that
\begin{equation*}
\eta(t)=\left\{
\begin{array}{ll}
0 & \text{ if }t\leq 0, \\
1 & \text{ if }t\geq 1,
\end{array}
\right. 
\quad\eta'(t)\geq 0\quad\forall\, t\in [0,1],
\end{equation*}
and, given $\delta>0$, define $\eta_\delta(t):=\eta\left(\frac{t-1}{\delta}\right)$. If $w\in\cal{D}^{1,p_i}(\R^N)$ then
\begin{equation}\label{5}
\eta_{\delta}\circ w\in {\cal D}^{1,p_i}(\R^N),\quad\nabla (\eta_{\delta}\circ w)=(\eta_{\delta}'\circ w)\nabla w,
\end{equation}
as a standard verification shows.\\
Now, fix $(\varphi,\psi)\in\cal{D}^{1,p_1}(\R^N)_+\times\cal{D}^{1,p_2}(\R^N)_+$. Multiply the first equation in \eqref{p} by $(\eta_\delta\circ u)\varphi$, integrate over $\R^N$, and use $({\rm H}_{f,g})$ to achieve
\begin{equation*}
\int_{\R^N}|\nabla u|^{p_1-2}\nabla u\nabla ((\eta_\delta\circ u)\varphi)\, dx
\leq M_1\int_{\R^N} a_1u^{\alpha_1}(1+v^{\beta _1})(\eta_\delta\circ u)\varphi\, dx.
\end{equation*}
By \eqref{5} we have
\begin{equation*}
\begin{split}
\int_{\R^N} & |\nabla u|^{p_1-2}\nabla u\nabla ((\eta_\delta\circ u)\varphi)\, dx\\
& =\int_{\R^N}|\nabla u|^{p_1}(\eta_\delta'\circ u)\varphi \,dx
+\int_{\R^N}(\eta_\delta\circ u)|\nabla u|^{p_1-2}\nabla u\nabla\varphi\,dx, 
\end{split}
\end{equation*}
while $\eta_\delta'\circ u\geq 0$ in $\R^N$. Therefore,
\begin{equation*}
\int_{\R^N}(\eta_\delta\circ u)|\nabla u|^{p_1-2}\nabla u\nabla\varphi\,dx\leq
M_1\int_{\R^N} a_1u^{\alpha_1}(1+v^{\beta _1})(\eta_\delta\circ u)\varphi\, dx.
\end{equation*}
Letting $\delta\to 0^+$ produces \eqref{trunc1}. The proof of \eqref{trunc2} is similar.
\end{proof}
\begin{lemma}[Moser's iteration]\label{L3}
There exists $R>0$ such that
\begin{equation}\label{apin2}
\max\{\Vert u\Vert_{L^\infty(\Omega_1)},\Vert v\Vert_{L^\infty(\Omega_2)}\}\leq R,
\end{equation}
where
$$\Omega_1:=\R^N(u>1)\quad\text{and}\quad\Omega_2:=\R^N(v>1),$$
for every $(u,v)\in \cal{D}^{1,p_1}(\R^N)\times\cal{D}^{1,p_2}(\R^N)$ solving problem \eqref{p}.
\end{lemma}
\begin{proof} Given $w\in L^p(\Omega_1)$, we shall write $\Vert w\Vert_p$ in place of $\Vert w\Vert_{L^p(\Omega_1)}$ when no confusion can arise. Observe that $m(\Omega_1)<+\infty$ and define, provided $M>1$,
$$u_M(x):=\min\{u(x),M\},\quad x\in\R^N.$$ 
Choosing $\varphi:=u_M^{\kappa p_1+1}$, with $\kappa\geq 0$, in \eqref{trunc1} gives
\begin{equation} \label{s41}
\begin{split}
(\kappa p_1+1)\int_{\Omega_1(u\leq M)} &  u_M^{\kappa p_1}\vert\nabla u\vert^{p_1-2}\nabla u\nabla u_M\, dx\\
& \leq M_1\int_{\Omega_1} a_1(1+v^{\beta_1}) u_M^{\kappa p_1+1}\, dx. 
\end{split}
\end{equation}
Through the Sobolev embedding theorem one has
\begin{equation*}
\begin{split}
& (\kappa p_1+1)\int_{\Omega_1(u\leq M)} u_M^{\kappa p_1}\vert\nabla u\vert ^{p_1-2}\nabla u\nabla u_M\, dx \\
& =(\kappa p_1+1)\int_{\Omega_1(u\leq M)}(|\nabla u|u^{\kappa})^{p_1}dx 
=\frac{\kappa p_1+1}{(\kappa +1)^{p_1}}\int_{\Omega_1(u\leq M)}|\nabla u^{\kappa+1}|^{p_1} dx \\
& =\frac{\kappa p_1+1}{(\kappa+1)^{p_1}}\int_{\Omega_1}|\nabla u_M^{\kappa+1}|^{p_1}dx
\geq C_1\frac{\kappa p_1+1}{(\kappa+1)^{p_1}}\Vert u_M^{\kappa+1}\Vert_{p_1^*}^{p_1}
\end{split}
\end{equation*}
for appropriate $C_1>0$. By Remark \ref{R1}, H\"{o}lder's inequality entails 
\begin{equation*}
\begin{split}
\int_{\Omega_1} a_1(1+v^{\beta_1})u_M^{\kappa p_1+1} dx & \leq\int_{\Omega_1}a_1(1+v^{\beta_1})u^{\kappa p_1+1}dx\\
& \leq\left(\Vert a_1 \Vert_{\xi_1}+ \Vert a_1 \Vert_{\zeta_1} \Vert v \Vert_{p_2^*}^{\beta_1}\right) 
\Vert u \Vert_{(\kappa p_1+ 1) \xi_1'}^{\kappa p_1+ 1}.
\end{split}
\end{equation*}
Hence, \eqref{s41} becomes
\begin{equation*}
\frac{\kappa p_1+1}{(\kappa+1)^{p_1}}\Vert u_M^{\kappa+1}\Vert_{p_1^*}^{p_1}\leq C_2\left(\Vert a_1\Vert_{\xi_1}+
\Vert a_1\Vert_{\zeta_1}\Vert v\Vert_{p_2^*}^{\beta_1}\right)\Vert u\Vert _{(\kappa p_1+1)\xi _1'}^{\kappa p_1+1}.
\end{equation*}
Since $u(x)=\displaystyle{\lim_{M\to\infty}}u_M(x)$ a.e. in $\R^N$, using the Fatou lemma we get
\begin{equation*}
\frac{\kappa p_1+1}{(\kappa+1)^{p_1}}\Vert u\Vert_{(\kappa+1)p_1^*}^{(\kappa+1)p_1}
\leq C_2\left(\Vert a_1\Vert_{\xi_1}+\Vert a_1\Vert_{\zeta_1}\Vert v\Vert_{p_2^*}^{\beta_1}\right)
\Vert u\Vert_{(\kappa p_1+1)\xi _1'}^{\kappa p_1+1},
\end{equation*}
namely
\begin{equation} \label{dis}
\Vert u\Vert_{(\kappa+1)p_1^*}\leq C_3^{\eta(\kappa)}\sigma(\kappa)
\left(1+\Vert v\Vert_{p_2^*}^{\beta_1}\right)^{\eta(\kappa)}
\Vert u\Vert_{(\kappa p_1+1)\xi _1'}^{\frac{\kappa p_1+1}{(\kappa+1)p_1}},
\end{equation}
where $C_3>0$, while
$$\eta(\kappa):=\frac{1}{(\kappa+1)p_1},\quad \sigma(\kappa):=\left[\frac{\kappa+1}{(\kappa p_1+1)^{1/p_1}}\right] ^{\frac{1}{\kappa +1}}.$$
Let us next verify that
$$(\kappa+1) p_1^* >(\kappa p_1+1) \xi_1'\quad\forall\, \kappa\in\R^+_0\, ,$$
which clearly means
\begin{equation}\label{cond-k}
\frac{1}{\xi_1}<1-\frac{\kappa p_1+ 1}{(\kappa+ 1) p_1^*},\quad\kappa\in\R^+_0\, .
\end{equation}
Indeed, the function $\kappa\mapsto\frac{\kappa p_1+1}{(\kappa+ 1) p_1^*}$ is increasing on $\R^+_0$ and tends to
$\frac{p_1}{p_1^*}$ as $k\to\infty$. So, \eqref{cond-k} holds true, because $\frac{1}{\xi_1}<1-\frac{p_1}{p_1^*}$; see Remark \ref{R1}. Now, Moser's iteration can start. If there exists a sequence $\{\kappa_n\}\subseteq\R^+_0$ fulfilling
\begin{equation*}
\lim_{n\to\infty}\kappa_n=+\infty,\quad\Vert u\Vert_{(\kappa_n+1)p_1^*}\leq 1\;\;\forall\, n\in\N
\end{equation*}
then $\Vert u\Vert_{L^\infty(\Omega_1)}\leq 1$. Otherwise, with appropriate $\kappa_0>0$, one has
\begin{equation}\label{kappa0}
\Vert u\Vert_{(\kappa+ 1)p_1^*}>1\;\;\text{for any}\;\;\kappa>\kappa_0,\;\;\text{besides}\;\;\Vert u \Vert_{(\kappa_0+ 1) p_1^*} \leq 1.
\end{equation}
Pick $\kappa_1>\kappa_0$ such that $(\kappa_1 p_1+1)\xi_1'=(\kappa_0+1) p_1^*$, set $\kappa:=\kappa_1$ in \eqref{dis}, and use \eqref{kappa0} to arrive at
\begin{equation}\label{uno}
\begin{split}
& \Vert u\Vert_{(\kappa_1 +1)p_1^*}\\
& \leq C_3^{\eta(\kappa_1)}\sigma(\kappa_1)\left(1+\Vert v\Vert_{p_2^*}^{\beta_1}\right)^{\eta(\kappa_1)}
\Vert u\Vert_{(\kappa_0+1)p_1^*}^{\frac{\kappa_1 p_1+1}{(\kappa_1+1)p_1}}\\
& \leq C_3^{\eta(\kappa_1)}\sigma(\kappa_1)\left(1+\Vert v\Vert _{p_2^*}^{\beta_1}\right)^{\eta(\kappa_1)}.
\end{split}
\end{equation}
Choose next $\kappa_2>\kappa_0$ satisfying $(\kappa_2 p_1+1)\xi_1'= (\kappa_1+1)p_1^*$. From \eqref{dis}, written for
$\kappa:= \kappa_2$, as well as \eqref{kappa0}--\eqref{uno}  it follows
\begin{equation*}
%\label{due}
\begin{split}
& \Vert u\Vert_{(\kappa_2 +1)p_1^*}\\
& \leq C_3^{\eta(\kappa_2)}\sigma(\kappa_2)\left(1+\left\Vert v\right\Vert_{p_2^*}^{\beta_1}\right))^{\eta(\kappa_2)}
\Vert u\Vert_{(\kappa_1+1)p_1^*}^{\frac{\kappa_2 p_1+1}{(\kappa_2 +1)p_1}}\\
& \leq C_3^{\eta(\kappa_2)}\sigma(\kappa_2)\left(1+\left\Vert v\right\Vert_{p_2^*}^{\beta_1}\right)^{\eta(\kappa_2)}
\Vert u\Vert_{(\kappa_1+1)p_1^*} \\
& \leq C_3^{\eta(\kappa_2)+\eta(\kappa_1)}\sigma(\kappa_2)\sigma(\kappa_1)\left(1+\Vert v\Vert_{p_2^*}^{\beta_1}\right)^{\eta(\kappa_2)+\eta(\kappa_1)}.
\end{split}
\end{equation*}
By induction, we construct a sequence $\{\kappa_n\}\subseteq(\kappa_0,+\infty)$ enjoying the properties below:
\begin{equation}\label{prop1}
(\kappa_n p_1+ 1)\xi_1'=(\kappa_{n-1}+ 1)p_1^*\, ,\quad n\in\N;
\end{equation}
\begin{equation}\label{prop2}
\Vert u\Vert_{(k_n+1)p_1^*}\leq C_3^{\sum_{i=1}^{n}\eta(\kappa_i)} \prod_{i=1}^{n}\sigma(\kappa_i)
\left(1+\Vert v\Vert_{p_2^*}^{\beta_1}\right)^{\sum_{i=1}^{n}\eta(\kappa_i)}
\end{equation}
 for all $n\in\N$. A simple computation based on \eqref{prop1} yields
\begin{equation}\label{asint}
(\kappa_n+1)\simeq(\kappa_0+1) \left(\frac{p_1^*}{p_1 \xi_1'}\right)^n\;\;\text{as}\;\; n\to\infty,
\end{equation}
where $\frac{p_1^*}{p_1\xi_1'}>1$ due to $({\rm j}_4)$ of Remark \ref{R1}. Further, if $C_4>0$ satisfies
$$1<\left[\frac{t+1}{(t p_1+1)^{1/p_1}}\right]^{\frac{1}{\sqrt{t +1}}}\leq C_4\, ,\quad t\in\R^+_0\, ,$$
(cf. \cite[p. 116]{DKN}) then
\begin{equation*}
\begin{split}
& \prod_{i=1}^{n}\sigma(\kappa_i)
=\prod_{i=1}^{n}\left[\frac{\kappa_i+1}{(\kappa_i p_1+1)^{1/p_1}}\right]^{\frac{1}{\kappa_i +1}}\\
& =\prod_{i=1}^{n}\left\{\left[\frac{\kappa_i+1}{(\kappa_i p_1+1)^{1/p_1}}\right]^{\frac{1}{\sqrt{\kappa_i +1}}}
\right\}^{\frac{1}{\sqrt{\kappa_i +1}}}\leq C_4^{\sum_{i=1}^{n}\frac{1}{\sqrt{\kappa_i +1}}}.
\end{split}
\end{equation*}
Consequently, \eqref{prop2} becomes
\begin{equation*}
\Vert u\Vert _{(k_n+1)p_1^*} \leq C_3^{\sum_{i=1}^{n}\eta(\kappa_i)} C_4^{\sum_{i=1}^{n} \frac{1}{\sqrt{\kappa_i+1}}}
\left(1+\Vert v \Vert_{p_2^*}^{\beta_1}\right)^{\sum_{i=1}^{n}\eta(\kappa_i)}.
\end{equation*}
Since, by \eqref{asint}, both $\kappa_n+1\to+\infty$ and $\frac{1}{\kappa_n+1}\simeq\frac{1}{\kappa_0+1}
\left(\frac{p_1\xi_1'}{p_1^*}\right)^n$, while \eqref{apin1} entails $\Vert v\Vert_{p^*_2}\leq\rho$, there exists a constant $C_5>0$ such that
\begin{equation*}
\Vert u\Vert_{(\kappa_n+1)p_1^*}\leq C_5\quad\forall\, n\in\N,
\end{equation*}
whence $\Vert u\Vert_{L^\infty(\Omega_1)}\leq C_5$. Thus, in either case, $\Vert u\Vert_{L^\infty(\Omega_1)}\leq R$,
with $R:=\max\{1,C_5\}$. A similar argument applies to $v$.
\end{proof}
Using \eqref{apin2}, besides the definition of sets $\Omega_i$, we immediately infer the following
\begin{theorem}\label{T3}
Under assumptions $({\rm H}_{f,g})$ and $({\rm H}_a)$, one has
\begin{equation}\label{ineqT3}
\max\{\Vert u\Vert_\infty,\Vert v\Vert_\infty\}\leq R
\end{equation}
for every weak solution $(u,v)\in \cal{D}^{1,p_1}(\R^N)\times \cal{D}^{1,p_2}(\R^{N})$ to problem \eqref{p}. Here, $R$ is given by Lemma \ref{L3}.
\end{theorem}
\section{The regularized system}
Assertion $({\rm j}_1)$ of Remark \ref{R1} ensures that $a_i\in L^{(p_i^*)'}(\R^N)$. Therefore, thanks to Minty-Browder's theorem \cite[Theorem V.16]{B}, the equation 
\begin{equation}\label{20}
-\Delta_{p_i} w_i=a_i(x)\quad\text{in}\quad\R^N
\end{equation}
possesses a unique solution $w_i\in\mathcal{D}^{1,p_i}(\R^N)$, $i=1,2$. Moreover, 
\begin{itemize}
\item $w_i>0$, and\\
\item $w_i\in L^\infty(\R^N)$.
\end{itemize}
Indeed, testing \eqref{20} with $\varphi:=w_i^-$ yields $w_i\geq 0$, because $a_i>0$ by $({\rm H}_a)$. Through the strong maximum principle we obtain
$$\essinf_{B_r(x)} w_i>0\;\;\text{for any } r> 0,\, x\in\R^N.$$
Hence, $w_i>0$. Moser's iteration technique then produces $w_i\in L^\infty(\R^N)$.

Next, fix $\eps\in\ ]0,1[$ and define
\begin{equation}\label{6}
(\underline{u},\underline{v})=\left([m_1(R+1)^{\alpha_1}]^{\frac{1}{p_1-1}}w_1,
[m_2(R+1)^{\beta_2}]^{\frac{1}{p_2-1}}w_2\right),  
\end{equation}
\begin{equation*}
(\overline{u}_\eps,\overline{v}_\eps)=\left([ M_1\eps ^{\alpha_1}(1+R^{\beta_1})]^{\frac{1}{p_1-1}}w_1,
[M_2\eps^{\beta_2}(1+R^{\alpha_2})]^{\frac{1}{p_2-1}}w_2\right), 
\end{equation*}
as well as
\begin{equation*}
\mathcal{K}_\eps:=\left\{ (z_1,z_2)\in L^{p_1^*}(\R^N)\times L^{p_2^*}(\R^N):\underline{u}\leq z_1\leq\overline{u}_\eps\, ,\;\underline{v}\leq z_2\leq\overline{v}_\eps\right\}.
\end{equation*}
Obviously, $\mathcal{K}_\eps$ is bounded, convex, closed  in $L^{p_1^*}(\R^N)\times L^{p_2^*}(\R^N)$. Given $(z_1,z_2)\in \mathcal{K}_\eps$, write
\begin{equation}\label{10}
\tilde{z}_i:=\min\{z_i, R\},\quad i=1,2.
\end{equation}
Since, on account of \eqref{10}, hypothesis $({\rm H}_{f,g})$ entails
\begin{eqnarray}\label{9}
a_1 m_1 (R+1)^{\alpha_1}\leq a_1 f(\tilde{z}_1+\eps ,\tilde{z}_2)\leq a_1 M_1\eps^{\alpha_1}(1+R^{\beta_1}),\nonumber\\
\phantom{}\\
a_2 m_2(R+1)^{\beta_2}\leq a_2 g(\tilde{z}_1,\tilde{z}_2+\eps)\leq a_2 M_2 (1+R^{\alpha_2})\eps ^{\beta_2}\nonumber,
\end{eqnarray}
while, recalling Remark \ref{R1}, $a_i\in L^{(p_i^*)'}(\R^N)$, the functions 
$$x\mapsto a_1(x) f(\tilde{z}_1(x)+\eps ,\tilde{z}_2(x)),\quad x\mapsto a_2(x) g(\tilde{z}_1(x),\tilde{z}_2(x)+\eps)$$
belong to $\mathcal{D}^{-1,p_1'}(\R^N)$ and $\mathcal{D}^{-1,p_2'}(\R^N)$, respectively.  Consequently, by Minty-Browder's theorem again, there exists a unique weak solution $(u_\eps,v_\eps)$ of the problem
\begin{equation}\label{prr}
\left\{ 
\begin{array}{ll}
-\Delta_{p_1} u=a_1(x) f(\tilde{z}_1(x)+\eps ,\tilde{z}_2(x)) & \text{in }\R^N, \\
-\Delta _{p_2} v=a_2(x) g(\tilde{z}_1(x),\tilde{z}_2(x)+\eps) & \text{in }\R^N, \\
u_\eps,v_\eps>0 & \text{in }\R^N.
\end{array}
\right. 
\end{equation}
Let $\mathcal{T}:\mathcal{K}_\eps\to  L^{p_1^*}(\R^N)\times L^{p_2^*}(\R^N)$ be defined by $\mathcal{T}(z_1,z_2)=(u_\eps ,v_\eps)$ for every $(z_1,z_2)\in\mathcal{K}_\eps$.
\begin{lemma}\label{L1}
One has $\underline{u}\leq u_\eps\leq \overline{u}_\eps$ and $\underline{v}\leq v_\eps\leq\overline{v}_\eps$. So, in particular, $\mathcal{T}(\mathcal{K}_\eps)\subseteq\mathcal{K}_\eps$.
\end{lemma}
\begin{proof}
Via \eqref{6}, \eqref{20}, \eqref{prr}, and \eqref{9} we get
\begin{equation*}
\begin{split}
& \int_{\R^N}(-\Delta_{p_1}\underline{u}-(-\Delta_{p_1}u_\eps) )(\underline{u}-u_\eps)^+ dx\\
& =\int_{\R^N}\left(-\Delta_{p_1}(\{m_1 (R+1)^{\alpha_1}\}^{\frac{1}{p_1-1}}w_1)-(-\Delta_{p_1}u_\eps )\right)
(\underline{u}-u_\eps)^+ dx \\ 
& =\int_{\R^N} a_1\left((m_1(R+1)^{\alpha _1}-f(\tilde{z}_1+\eps,\tilde{z}_2)\right) (\underline{u}-u_\eps)^+ dx\leq 0.
\end{split}
\end{equation*}
Lemma A.0.5 of \cite{P} furnishes
\begin{equation*}
\begin{split}
& \int_{\R^N}(-\Delta_{p_1}\underline{u}-(-\Delta_{p_1}u_\eps) )(\underline{u}-u_\eps)^+ dx\\
& =\int_{\R^N}\left(|\nabla\underline{u}|^{p_1-2}\nabla\underline{u}-|\nabla u_\eps|^{p_1-2}\nabla u_\eps\right)
\nabla (\underline{u}-u_\eps)^+ dx\geq 0.
\end{split}
\end{equation*}
Therefore,
\begin{equation*}
\int_{\R^N}\left(-\Delta_{p_1}\underline{u}-(-\Delta_{p_1} u_\eps)\right) (\underline{u}-u_\eps)^+ dx=0,
\end{equation*}
which implies $(\underline{u}-u_\eps)^+=0$, i.e., $\underline{u}\leq u_\eps$. The remaining inequalities can be verified in a similar way.
\end{proof}
\begin{lemma}\label{L2}
The operator $\mathcal{T}$ is continuous and compact.
\end{lemma}
\begin{proof}
Pick a sequence $\{(z_{1,n},z_{2,n})\}\subseteq\mathcal{K}_\eps$ such that
$$(z_{1,n},z_{2,n})\to (z_1,z_2)\quad\text{in}\quad L^{p_1^*}(\R^N)\times L^{p_2^*}(\R^N).$$
If $(u_n,v_n):=\mathcal{T}(z_{1,n},z_{2,n})$ and $(u,v):=\mathcal{T}(z_1,z_2)$ then
\begin{eqnarray}
\int_{\R^N}|\nabla u_n|^{p_1-2}\nabla u_n\nabla \varphi\, dx  
=\int_{\R^N } a_1 f(\tilde{z}_{1,n}+\eps,\tilde{z}_{2,n})\varphi\, dx, \label{equa1}\\
\int_{\R^N}|\nabla v_n|^{p_2-2}\nabla v_n\nabla \psi\, dx
=\int_{\R^N} a_2 g(\tilde{z}_{1,n},\tilde{z}_{2,n}+\eps)\psi\, dx,\label{equa1b}\\
\int_{\R^N}|\nabla u|^{p_1-2}\nabla u\nabla\varphi\, dx=\int_{\R^N} a_1 f(\tilde{z}_1+\eps,\tilde{z}_2)\varphi\, dx,\nonumber\\
\int_{\R^N}|\nabla v|^{p_2-2}\nabla v\nabla\psi\, dx =\int_{\R^N} a_2 g(\tilde{z}_1,\tilde{z}_2+\eps)\psi\, dx\nonumber
\end{eqnarray}
for every $(\varphi,\psi )\in\mathcal{D}^{1,p_1}(\R^N)\times\mathcal{D}^{1,p_2}(\R^N)$. Set $\varphi:=u_n$ in \eqref{equa1}. From \eqref{9} it follows, after using H\"{o}lder's inequality,
\begin{equation*}
\begin{split}
\Vert \nabla u_n\Vert_{p_1}^{p_1} & =\int_{\R^N} a_1 f(\tilde{z}_{1,n}+\eps,\tilde{z}_{2,n}) u_n\, dx \\
& \leq M_1\int_{\R^N} a_1 \eps^{\alpha_1}(1+R^{\beta_1})u_n\, dx\leq C_\eps\int_{\R^N}a_1 u_n\, dx \\
& \leq C_\eps\Vert a_1\Vert_{(p_1^*)'} \Vert u_n\Vert_{p_1^*}\leq C_\eps\Vert a_1\Vert_{(p_1^*)'}
\Vert \nabla u_n\Vert_{p_1}\;\;\forall\, n\in\N,
\end{split}
\end{equation*}
where $C_\eps:=M_1\eps^{\alpha_1}(1+R^{\beta_1})$. This actually means that $\{u_n\}$ is bounded in $\mathcal{D}^{1,p_1}(\R^N)$, because $p_1>1$. By \eqref{equa1b}, an analogous conclusion holds for $\{v_n\}$. Along subsequences if necessary, we may thus assume
\begin{equation}\label{wone}
(u_n,v_n)\rightharpoonup (u,v)\;\;\text{in}\;\;\mathcal{D}^{1,p_1}(\R^N)\times\mathcal{D}^{1,p_2}(\R^N).
\end{equation}
So, $\{(u_n,v_n)\}$ converges strongly in $L^{q_1}(B_{r_1})\times L^{q_2}(B_{r_2})$ for any $r_i>0$ and any $1\leq q_i\leq p_i^*$, whence, up to subsequences again,
\begin{equation}\label{wtwo}
(u_n,v_n)\to (u,v)\;\;\text{a.e. in}\;\;\R^N.
\end{equation}
Now, combining Lemma \ref{L1} with Lebesgue's dominated convergence theorem, we obtain
\begin{equation}\label{convpair}
(u_n,v_n)\to (u,v)\;\;\text{in}\;\; L^{p_1^*}(\R^N)\times L^{p_2^*}(\R^N),
\end{equation}
as desired. Let us finally verify that $\mathcal{T}(\cal{K}_\eps)$ is relatively compact. If $(u_n,v_n):=\mathcal{T}(z_{1,n},z_{2,n})$, $n\in\N$, then \eqref{equa1}--\eqref{equa1b} can be written. Hence, the previous argument yields a pair $(u,v)\in L^{p_1^*}(\R^N)\times L^{p_2^*}(\R^N)$ fulfilling \eqref{convpair}, possibly along a subsequence. This completes the proof.
\end{proof}
Thanks to Lemmas \ref{L1}--\ref{L2}, Schauder's fixed point theorem applies, and there exists $(u_\eps,v_\eps)\in\mathcal{K}_\eps$ such that $(u_\eps,v_\eps)=\mathcal{T}(u_\eps,v_\eps)$. Through Theorem \ref{T3}, we next arrive at
\begin{theorem}\label{T2}
Under hypotheses $({\rm H}_{f,g})$ and $({\rm H}_a)$, for every $\eps>0$ small, problem \eqref{pr} admits a solution
$\left(u_\eps,v_\eps\right)\in\mathcal{D}^{1,p_1}(\R^N)\times\mathcal{D}^{1,p_2}(\R^N)$ complying with \eqref{ineqT3}.
\end{theorem}
\section{Existence of solutions}\label{S4}
We are now ready to establish the main result of this paper.
\begin{theorem}\label{mainthm}
Let $({\rm H}_{f,g})$ and $({\rm H}_a)$ be satisfied. Then \eqref{p} has a weak solution $(u,v)\in\mathcal{D}^{1,p_1}(\R^N)\times\mathcal{D}^{1,p_2}(\R^N)$, which is essentially bounded.
\end{theorem}
\begin{proof}
Pick $\eps:=\frac{1}{n}$, with $n\in\N$ big enough. Theorem \ref{T2} gives a pair $(u_n,v_n)$, where $u_n:=u_{\frac{1}{n}}$ and $v_n:=v_{\frac{1}{n}}$, such that
\begin{equation}\label{122}
\begin{split}
\int_{\R^N}\vert\nabla u_n\vert^{p_1-2}\nabla u_n\nabla\varphi \, dx & =\int_{\R^N} a_1 f(u_n+\frac{1}{n},v_n)\varphi\, dx, \\ 
\int_{\R^N}\vert\nabla v_n\vert^{p_2-2}\nabla v_n\nabla\psi\, dx & =\int_{\R^N}a_2 g(u_n,v_n+\frac{1}{n})\psi\, dx
\end{split}
\end{equation}
for every $(\varphi,\psi )\in\mathcal{D}^{1,p_1}(\R^N)\times\mathcal{D}^{1,p_2}(\R^N)$, as well as (cf. Lemma \ref{L1})
\begin{equation}\label{13}
0<\underline{u}\leq u_n\leq R,\quad 0<\underline{v}\leq v_n\leq R.
\end{equation}
Thanks to $({\rm H}_{f,g})$, \eqref{13}, and $({\rm H}_a)$, choosing $\varphi:=u_n$, $\psi:=v_n$ in \eqref{122} easily entails
\begin{equation*}
\begin{split}
\Vert\nabla u_n\Vert_{p_1}^{p_1} & \leq M_1\int_{\R^N} a_1 u_n^{\alpha_1+1}(1+v_n^{\beta_1}) dx
\leq M_1 R^{\alpha_1+1}(1+R^{\beta_1})\Vert a_1\Vert_1\, ,\\
\Vert\nabla v_n\Vert_{p_2}^{p_2} & \leq M_2\int_{\R^N} a_2 (1+u_n^{\alpha_2})v_n^{\beta_2+1} dx 
\leq M_2(1+R^{\alpha_2})R^{\beta_2+1}\Vert a_2\Vert_1,
\end{split}
\end{equation*}
whence both $\{u_n\}\subseteq\mathcal{D}^{1,p_1}(\R^N) $ and $\{v_n\}\subseteq\mathcal{D}^{1,p_2}(\R^N)$ are bounded. Along subsequences if necessary, we thus have \eqref{wone}--\eqref{wtwo}. Let us next show that 
\begin{equation}\label{15}
(u_n,v_n)\to (u,v)\quad\text{strongly in}\quad\mathcal{D}^{1,p_1}(\R^N)\times\mathcal{D}^{1,p_2}(\R^N).
\end{equation}
Testing the first equation in \eqref{122} with $\varphi:=u_n-u$ yields
\begin{equation}\label{new1}
\int_{\R^N}\vert \nabla u_n\vert ^{p_1-2}\nabla u_n\nabla (u_n-u) dx=\int_{\R^N} a_1f(u_n+\frac{1}{n},v_n)(u_n-u) dx.
\end{equation}
The right-hand side of \eqref{new1} ges to zero as $n\to\infty$. Indeed, by $({\rm H}_{f,g})$, \eqref{13}, and $({\rm H}_a)$ again,
\begin{equation*}
|a_1f(u_n+\frac{1}{n},v_n)(u_n-u)|\leq 2M_1R^{\alpha_1+1}(1+R^{\beta_1})a_1\quad\forall\, n\in\N,
\end{equation*}
so that, recalling \eqref{wtwo}, Lebesgue's dominated convergence theorem applies. Through \eqref{new1} we obtain
$\displaystyle{\lim_{n\to\infty}}\langle -\Delta_{p_1}u_n,u_n-u\rangle=0$. Likewise, $\langle-\Delta _{p_2}v_n,v_n-v\rangle\to 0$
as $n\to\infty$, and \eqref{15} directly follows from Proposition \ref{Sprop}. On account of \eqref{122}, besides \eqref{15}, the final step is to verify that
\begin{equation} \label{right}
\lim_{n\to\infty}\int_{\R^N} a_1 f(u_n+\frac{1}{n},v_n)\varphi\, dx=\int_{\R^N} a_1 f(u,v)\varphi\, dx,
\end{equation}
\begin{equation}\label{right2}
\lim_{n\to\infty}\int_{\R^N} a_2 g(u_n,v_n+\frac{1}{n})\psi\, dx=\int_{\R^N} a_2 g(u,v)\psi\, dx
\end{equation}
for all $(\varphi,\psi)\in\mathcal{D}^{1,p_1}(\R^N)\times\mathcal{D}^{1,p_2}(\R^N)$. If $\varphi\in\mathcal{D}^{1,p_1}(\R^N)$ then $({\rm j}_1)$ in Remark \ref{R1} gives $a_1\varphi\in L^1(\R^N)$. Since, as before,
\begin{equation*}
|a_1f(u_n+\frac{1}{n},v_n)\varphi|\leq M_1R^{\alpha_1+1}(1+R^{\beta_1})a_1|\varphi|,\quad n\in\N,
\end{equation*}
assertion \eqref{right} stems from the Lebesgue dominated convergence theorem. The proof of \eqref{right2} is similar at all.
\end{proof}
\section*{Acknowledgement}
This work is performed within the 2016--2018 Research Plan - Intervention Line 2: `Variational Methods and Differential Equations', and partially supported by GNAMPA of INDAM.
\end{document}